\documentclass{amsproc}

\usepackage{amssymb}
\usepackage{hyperref}

\newtheorem*{ta}{Theorem A}
\newtheorem*{pb}{Proposition B}
\newtheorem*{conj}{Conjecture}
\newtheorem{theorem}{Theorem}[section]
\newtheorem{lemma}[theorem]{Lemma}

\theoremstyle{definition}

\theoremstyle{remark}

\numberwithin{equation}{section}

\begin{document}

\title[Regular character-graphs]{Regular character-graphs \\whose eigenvalues are greater than or equal to -2}

\author[M. Ebrahimi]{Mahdi Ebrahimi}
\address{School of Mathematics, Institute for Research in Fundamental Sciences (IPM), Tehran, 19395-5746, Iran}
\email{m.ebrahimi.math@ipm.ir}
\thanks{This research was supported in part by a grant from School of Mathematics, Institute for Research in Fundamental Sciences (IPM)}

\author[M. khatami]{Maryam Khatami}
\address{Department of Pure Mathematics, Faculty of Mathematics and Statistics, University of Isfahan, Isfahan, 81746-73441, Iran}
\email{m.khatami@sci.ui.ac.ir}

\author[Z. Mirzaei]{Zohreh Mirzaei}
\address{Department of Pure Mathematics, Faculty of Mathematics and Statistics, University of Isfahan, Isfahan, 81746-73441, Iran}
\email{z.mirzaei@sci.ui.ac.ir}



\date{}

\dedicatory{}

\begin{abstract}
Let $G$ be a finite group and $\mathrm{Irr}(G)$ be the set of all complex irreducible characters of $G$. The character-graph $\Delta(G)$ associated to $G$, is a graph whose vertex set is the set of primes which divide the degrees of some characters in $\mathrm{Irr}(G)$ and two distinct primes $p$ and $q$ are adjacent in $\Delta(G)$ if the product $pq$ divides $\chi(1)$, for some  $\chi\in\mathrm{Irr}(G)$. Tong-viet posed the conjecture that if $\Delta(G)$ is  $k$-regular for some integer $k\geqslant 2$, then $\Delta(G)$ is either a complete graph or a cocktail party graph. In this paper,  we show that his conjecture is true for all regular character-graphs whose eigenvalues are in the interval $[-2, \infty )$.
\end{abstract}

\maketitle

\section{Introduction}
In this paper, all groups are assumed to be finite and all graphs are considered as finite simple graphs. Let $G$ be a group and $\mathrm{Out}(G)$ be the group of outer automorphisms of $G$. For some positive integer $n$, the set of prime divisors of $n$, is denoted by $\pi(n)$. Also  we use the notation $\pi(G)$ for $\pi(|G|)$. The set of all complex irreducible characters of $G$ is denoted by  $\mathrm{Irr}(G)$  and $\mathrm{cd}(G)=\{\chi(1)\, |\, \chi\in \mathrm{Irr}(G)\}$, is the set of degrees of irreducible characters  of $G$. The set of primes that divide some degrees in $\mathrm{cd}(G)$ is denoted by $\rho(G)$.

 The character-graph $\Delta(G)$ associated to $G$, is a graph with vertex set $\rho(G)$ and two distinct vertices $p$ and $q$ are adjacent in $\Delta(G)$, if the product $pq$ divides some character degrees in $\mathrm{cd}(G)$ \cite{manz}. In last decades, studying character-graphs has been an interesting field of research for mathematicians. For example, Lewis classified the structure of finite solvable groups with disconnected character-graphs \cite{lewis.dis}. Also  Akhlaghi et al. proved that  the complement of the character-graph of a solvable group is bipartite \cite{a1}. As another example, the first author showed that if the diameter of $\Delta(G)$ is equal to three, then the complement of $\Delta(G)$ is bipartite \cite{e.diam}. For a survey on this topic we refer the readers to \cite{l1}.

Studying eigenvalues of graphs, is an attractive field of research for graph theorists. For example, it has been proved that a connected regular graph of degree $r> 0$, is strongly regular if and only if it has exactly three distinct eigenvalues \cite[Theorem $3.6.4$]{spect}. Also  characterizing graphs with bounded smallest eigenvalue, is one of the important research problems (see for example  \cite{cameron, doob, hoffman}). Determining graphs whose least eigenvalue is greater than or equal to $-2$, has received much attention for many years. In \cite{cameron}, Cameron et al. characterized this class of graphs completely. They proved that graphs with eigenvalues in the interval  $[-2,\infty )$ are either  generalized line graphs or some connected graphs which are called exceptional graphs. Regularity of graphs is also one of the important properties of graphs which can always be recognized using the eigenvalues of graphs. Specially the largest eigenvalue of regular graphs  would be the degree of regularity \cite[Corollary $3.2.2$]{spect}. Among all regular graphs, those whose eigenvalues are in the interval $[-2,\infty )$ have their own importance. For example, it has been proved that  a connected regular graph with least eigenvalue greater than $-2$, is either a complete graph or an odd cycle \cite[Theorem $2. 5$]{doob}. Also it has been shown that a regular graph with least eigenvalue  $-2$, is either a generalized line graph or one of the $187$  exceptional graphs, which are classified in three definite layers {\rm(}\cite{nn}, \cite[Theorem $4.1.5$]{spectb1}{\rm)}. As another example, Abdollahi et al.  studied distance-regular Cayley graphs whose least eigenvalue is equal to $-2$ \cite{abdollahi}.

In last decade, regular character-graphs have been also studied in some details. For example, Tong-Viet proved that the character-graph $\Delta(G)$ of a finite group $G$ is $3$-regular if and only if $\Delta(G)\cong K_4$ \cite[Theorem $A$]{tongreg}. Also the following conjecture has been posed by him.
\begin{conj}
 Let $G$ be a group. If $\Delta(G)$ is $k$-regular,  for some integer  $k\geqslant 2$, then $\Delta(G)$ is either a complete graph of order $k+1$ or a cocktail party graph of order $k+2$. 
\end{conj}
Morresi-Zuccari showed that Tong-viet's conjecture is true for solvable groups \cite[Theorem $A$]{morre.reg}. Also in \cite[Theorem $A$]{akhlaghi.reg}, Sayanjaly et al.  proved that  a regular character-graph $\Delta(G)$ with odd order,  of a group $G$, is a complete graph, which confirms the Tong-viet's conjecture in some cases. In this paper, we wish to prove that Tong-viet's conjecture is  true for all regular character-graphs whose eigenvalues are in the interval $[-2,\infty )$. Now we are ready to state our main results.

\begin{ta}
Let $G$ be a finite group and $k\geqslant 2$ be an integer. Suppose that  $\Delta(G)$ is a $k$-regular character-graph whose eigenvalues are in the interval $[-2,\infty )$. Then   $\Delta(G)$ is isomorphic to either a complete graph of order $k+1$, or a cocktail party graph $\mathrm{CP}(m)$, where for some positive integer $m$ we have $|\rho(G)|=k+2=2m$.
\end{ta}
To prove Theorem A, we state a simple but interesting observation on non-abelian simple groups which is a fundamental tool for us in this paper.
\begin{pb}
Let $S$ be a non-abelian simple group. Then  $|\pi(\mathrm{Out}(S))\setminus \pi(S)|< |\pi(S)|$.
\end{pb}

\section{Preliminaries and Notations}\label{sec2}
In this section, we state some information and notations which we will use in the next sections. Let $G$ be a group and $R(G)$ the  solvable radical of $G$. For $H\leq\, G$ and $\theta\in\mathrm{Irr}(H)$, the set of all irreducible characters of $G$ lying over $\theta$ is denoted by $\mathrm{Irr}(G|\theta)$ and $\mathrm{cd}(G|\theta)=\{\chi(1)\, \big|\, \chi\in \mathrm{Irr}(G|\theta) \}$. Also if $N\lhd G$ and $\theta\in\mathrm{Irr}(N)$, the inertia subgroup of $\theta$ in $G$ is denoted by $I_{G}(\theta)$. We frequently use, Clifford's theorem  \cite[Theorem $6.11$]{i} and Gallagher's theorem \cite[Corollary $6.17$]{i}, without any further references. Also we use Zsigmondy's theorem which can be found in \cite{zsig}. We begin with the following lemmas.
\begin{lemma}[{\cite[Lemma $11.29$]{i}}]\label{11.29}
Let $N\unlhd\, G$ and $\phi\in\mathrm{Irr}(N)$. Then for every $\chi\in\mathrm{Irr}(G\, |\, \phi)$, $\frac{\chi(1)}{\phi(1)}\, \Big| \, |G:N|$.
\end{lemma}

\begin{lemma}[{\cite[Lemma $4.2$]{t.notriangle}}]\label{notri}
Let $N$ be a normal subgroup of a group $G$ such that $\frac{G}{N}\cong S$, where $S$ is a non-abelian simple group. Also let $\theta\in\mathrm{Irr}(N)$. Then either $\frac{\chi(1)}{\theta(1)}$ is divisible by two distinct primes in $\pi(S)$ for some $\chi\in\mathrm{Irr}(G|\theta)$ or $\theta$ is extendible to $\theta_{0}\in\mathrm{Irr}(G)$ and $S\cong A_{5}$ or $\mathrm{PSL}_{2}(8)$.
\end{lemma}

Now we use \cite{bondy}, \cite{spect} and \cite{spectb1} to review some necessary concepts in graph theory. Let $\Gamma$ be a finite graph with vertex set $V(\Gamma)$ and edge set $E(\Gamma)$. The number $|V(\Gamma)|$ is called the order of  $\Gamma$. If $E(\Gamma)=\emptyset$, then $\Gamma$ is called an empty graph. We use the notation $vw$, for the edge between  $v, w\in V(\Gamma)$. The complement of $\Gamma$ is denoted by  $\overline{\Gamma}$.  Also for $X\subseteq\, V(\Gamma)$, the induced subgraphs of $\Gamma$ on the subsets $X$ and $V(\Gamma ) \setminus X$, are denoted by  $\Gamma[X]$ and $\Gamma\setminus X$, respectively. A cut edge of $\Gamma$ is an edge whose removal from the graph $\Gamma$, increases the number of connected components of $\Gamma$. The clique number of $\Gamma$, denoted by $\omega (\Gamma )$, is the order of a largest complete subgraph of $\Gamma$. Also a matching in the graph $\Gamma$, is a set of pairwise non-adjacent edges. A maximum matching in $\Gamma$ is a matching with maximum number of edges among all matchings in $\Gamma$ and the matching number of $\Gamma$, denoted by $\alpha^{'}(\Gamma )$, is the size of a maximum matching. Now let $\Delta$ be a graph with vertex set $V(\Delta)$ for which $V(\Gamma)\cap V(\Delta)=\emptyset$. The join graph $\Gamma\star \Delta$ is the union of  $\Gamma$ and $ \Delta$  together with all edges joining $V(\Gamma)$ and $V(\Delta)$. For each vertex $v\in V(\Gamma)$, the degree of $v$ is denoted by $deg(v)$. If $deg(v)=r$, for every $v\in V(\Gamma)$,  then  $\Gamma$ is called an $r$-regular graph. The complete graph and the cycle with $n$ vertices are denoted by $K_n$ and $C_n$, respectively. For some positive integer $n$, the complement of the disjoint union of $n$  copies of $K_2$ is called the cocktail party graph which is denoted by $\mathrm{CP}(n)$. A bipartite graph whose partitions have $n_1$ vertices of degree $r_1$ and $n_2$ vertices of degree $r_2$ such that $n_1 r_1 = n_2 r_2$, is called a semi-regular bipartite graph  with parameters $(n_1, n_2, r_1, r_2)$. The complete bipartite graph whose partitions have  $m$ and $n$ vertices is denoted by  $K_{m,n}$.

Let $V(\Gamma):=\{v_1, v_2, \ldots , v_n\}$. The adjacency matrix of $\Gamma$ is the $n\times n$ matrix, denoted by $A$ ($A(\Gamma )$), whose rows and columns are indexed by $V(\Gamma )$ so that the $(v_i,v_j)$-entry $a_{ij}$ is equal to $1$, if $v_i v_j\in E(\Gamma )$ and $0$ otherwise, where $1 \leqslant i, j \leqslant n$. The eigenvalue of $\Gamma$, is an eigenvalue of the adjacency matrix of $\Gamma$. Since the adjacency matrix of $\Gamma$ is symmetric, the eigenvalues of $\Gamma$ are real numbers. The line graph $L(\Gamma)$ of the graph $\Gamma$, is a graph whose vertex set is $E(\Gamma)$ so that two vertices in $L(\Gamma)$ are adjacent whenever the corresponding edges in $\Gamma$ have exactly one vertex in common. Let  $a_1, \ldots , a_n$ be non-negative integers. The generalized line graph $L(\Gamma; a_1, \ldots , a_n)$, consists of disjoint union of $L(\Gamma)$ and $\mathrm{CP}(a_i)$, for  $i = 1, 2, \ldots , n $, along with all edges joining a vertex $v_i v_j $ of $L(\Gamma)$ with each vertex in $\mathrm{CP}(a_i)$ and $\mathrm{CP}(a_j)$, for every distinct $i, j\in\{1, \ldots , n\}$.   As special cases, line graphs ($a_1 = a_2 = \ldots = a_n=0$) and the cocktail party graph $\mathrm{CP}(n)=L(K_1; n)$, are generalized line graphs. Note that eigenvalues of the generalized line graphs are in the interval  $[-2,\infty)$. An exceptional graph is a connected graph, other than a generalized line graph, whose eigenvalues are in the interval $[-2,\infty )$.

Given a subset  $U$ of $V(\Gamma )$, the switching of $\Gamma$ with respect to $U$, denoted by $\Gamma_{U}$, is a graph with vertex set $V(\Gamma )$ and the adjacency between it's vertices is defined as follows:  $\Gamma_{U}[U]=\Gamma[U]$, $\Gamma_{U}\setminus  U=\Gamma\setminus U$ and for every $u\in U$ and $v\in V(\Gamma ) \setminus U$, the vertices $u$ and $v$ are adjacent in $\Gamma_{U}$ as well as they are adjacent in  $\overline{\Gamma}$. The Schl\"{a}fli graph and the Clebsch graph are two important examples of exceptional graphs which are defined as follows. Let $\Gamma =L(K_8)$  and $U\subseteq V(\Gamma )$ be the set of neighbours of a given vertex  $v\in V(\Gamma )$. Thus $\Gamma_{U}$ is a graph in which $v$ is an isolated vertex. Then $\Gamma_{U}\setminus \{v\}$ is a $16$-regular graph with $27$ vertices, which is called the Schl\"{a}fli graph and is denoted by $Sch_{27}$. Note that $\omega(Sch_{27})=6$. Now let $\Gamma =L(K_{4,4})$ and $U\subseteq V(\Gamma )$ be the set of vertices of a subgraph of $\Gamma$ isomorphic to  $L(K_{4,2})$. Then $\Gamma_{U}$ is a $10$-regular graph with $16$ vertices called the Clebsch graph.

\begin{lemma}\label{lineh}
Let $\Delta$ be a simple graph. Also let $\Delta=L(\Gamma)$, for some graph $\Gamma$ and $\overline{\Delta}$  be bipartite. Then the following hold:\\
{\bf 1.} $\alpha^{'}(\Gamma)\leqslant 2$ and, \\
{\bf 2.} $\Gamma$ has no subgraph isomorphic to $C_{5}$.
\end{lemma}
\begin{proof}
{\bf 1.} On the contrary, assume that $\alpha^{'}(\Gamma)\geqslant 3$. Then as  $\overline{\Delta}=\overline{L(\Gamma)}$, $\overline{\Delta}$ contains a copy of $K_3$ and it is a contradiction with this fact that $\overline{\Delta}$ is bipartite. \\
{\bf 2.}  On the contrary suppose that $\Gamma$ has a subgraph isomorphic to $C_{5}$ with edge set $A$. Then $\Delta [A]=L(\Gamma)[A]\cong C_{5}$. This is a contradiction as $\overline{\Delta}$ is bipartite.
\end{proof}


 In the following lemmas, regular connected graphs whose eigenvalues are in the interval $[-2, \infty )$, have been verified.
\begin{lemma}[{\cite[Proposition $1.1.5$]{spectb1}}]\label{reg.lin.con}
Let for some graph $\Gamma$, the line graph $L(\Gamma)$ be  regular and connected. Then the graph $\Gamma$ is either regular or semi-regular bipartite.
\end{lemma}

 \begin{lemma}[{\cite[Proposition $1.1.9$]{spectb1}}]\label{reg.gen.con}
Let $\Gamma$ be a regular connected generalized line graph. Then $\Gamma$ is either a cocktail party graph or a line graph.
\end{lemma}

\begin{lemma}[{\cite[Theorem $4.1.5$]{spectb1}}]\label{exceptional}
Let $\Gamma$ be an $r$-regular exceptional graph with $n$ vertices, where $r$ and $n$ are positive integers. Then one of the following layers occurs:\\
{\bf a) } $n=2(r+2)\leqslant 28$, \\
{\bf b) } $n=\frac{3}{2}(r+2)\leqslant 27$ and $\Gamma$ is an induced subgraph of the Schl\"{a}fli graph, or  \\
{\bf c) } $n=\frac{4}{3}(r+2)\leqslant 16$ and $\Gamma$ is an induced subgraph of the Clebsch graph.
\end{lemma}

\begin{lemma}\label{exceptional2}
Suppose $\Gamma$ is an $r$-regular exceptional graph with $n$ vertices, where $r$ and $n$ are positive integers. Then $\overline{\Gamma}$ is non-bipartite.
\end{lemma}
\begin{proof}
Using Lemma \ref{exceptional}, one of the following occurs:\\
{\bf 1.} $\Gamma$ lies in the first layer. Then as $n=2(r+2)$, it is clear that $\overline{\Gamma}$ is non-bipartite. \\
{\bf 2.}  $\Gamma$ is a graph in the second layer. Then as $n=\frac{3}{2} (r+2)$ and $\omega (\Gamma )\leqslant \omega (Sch_{27})=6$, we deduce that $\overline{\Gamma}$ is non-bipartite.\\
{\bf 3.}   $\Gamma$ is a graph in the third layer. There exist exactly three graphs in this layer, given in \cite[ Proposition $8.4$ ]{nn}. Thus it is easy to see that $\Gamma$ has an induced subgraph isomorphic to $C_5$ and hence $\overline{\Gamma}$ is non-bipartite. This completes the proof.
\end{proof}

Now we state some results on character-graphs of a finite group $G$.

\begin{lemma}[{\cite[Theorem $1.2$]{akhlaghi.reg}}]\label{akh.reg}
Let $G$ be a finite group and $|\rho(G)|\geqslant 3$. The character-graph $\Delta(G)$  is $k$-regular for some integer $k$ if and only if either $\Delta(G)\cong \overline{K_3}$ or $\overline{\Delta(G)}$ is a bipartite graph.
\end{lemma}

\begin{lemma}\label{simcutedge}
Let $S$ be a non-abelian simple group. Then either $\Delta(S)$ is connected and has no cut edges or $S$ is isomorphic to one of the groups $\mathrm{M}_{11}$, $\mathrm{PSL}_{3}(4)$, $^{2}\mathrm{B}_{2}(q^2)$,  where $q^2=2^{2m+1}$ for some integer $m\geqslant 1$ and $r:=q^2-1$ is a Mersenne prime and  $\mathrm{PSL}_{2} (q)$,  where $q$ is a prime power.
\end{lemma}
\begin{proof}
Suppose that $\Delta(S)$ is disconnected or contains at least one cut edge. If $\Delta(S)$ is disconnected, then using \cite{whitesim}, $S\cong \mathrm{PSL}_{2}(q)$ where $q$ is a prime power. Now assume that $\Delta(S)$ is connected and has at least one cut edge. Then as $\Delta(S)$ has at least one cut edge, using \cite{whitesim}, it is easy to see that $S$ is isomorphic to $\mathrm{M}_{11}$, $\mathrm{PSL}_{3}(4)$ or $^{2}\mathrm{B}_{2}(q^2)$ where $q^2=2^{2m+1}$ for some positive integer $m$. Since $|\pi(q^4+1)|\geqslant 2$, we deduce that $r:=q^2-1$ is a Mersenne prime and edge between $2$ and $r$ is a cut edge. This completes the proof.
\end{proof}

Now we state some properties of  $\mathrm{PSL}_2(q)$ where $q$ is a prime power. We will use Dickson's list of the subgroups of $\mathrm{PSL}_2(q)$, which can be found  in  {\cite[$II.8.27$]{huppert.group}}. The structure of the character-graph of  $\mathrm{PSL}_2(q)$ has been determined as follows.
\begin{lemma}[{\cite{whitesim}}]\label{deltaS}
Let $S\cong\mathrm{PSL}_2(q)$, where $q=u^{\alpha}\geqslant 4$  is a prime power.\\
{\bf 1.} If $q$ is even, then $\Delta(S)$ has three connected components, $\{2\}$, $\pi(q-1)$ and $\pi(q+1)$, and each component is a complete graph.\\
{\bf 2.} If $q> 5$ is odd, then $\Delta(S)$ has two connected components $\{u\}$ and $\pi((q-1)(q+1))$.\\
{\bf a.}  The connected component $\pi((q-1)(q+1))$ is a complete graph if and only if $q-1$ or $q+1$ is a power of $2$.\\
{\bf b.} If neither of $q-1$ or $q+1$ is a power of $2$, then $\pi((q-1)(q+1))$ can be partitioned as $\{2\}\cup M\cup P$, where $M=\pi(q-1)\setminus\{2\}$ and $P=\pi(q+1)\setminus\{2\}$ are both nonempty sets. The subgraph of $\Delta(S)$ corresponding to each of the subsets $M$ and $P$ is complete, all primes are adjacent to $2$, and no prime in $M$ is adjacent to any prime in $P$.
\end{lemma}

\begin{lemma}[{\cite[Lemma $4.5$]{l5}}]\label{index}
Let $S\cong\mathrm{PSL}_2(q)$ where $q=u^{\alpha}\geqslant\, 5$, such that  $u$ is a prime, $\alpha\geqslant\, 2$, and $q\neq\, 9$. If $S\leq\, G\leq\, \mathrm{Aut}(S)$, then $G$ has irreducible characters of degrees $(q+1)|G:G\cap\mathrm{PGL}_2(q)|$ and $(q-1)|G:G\cap\mathrm{PGL}_2(q)|$.
\end{lemma}

\begin{lemma}[{\cite[Lemma $3.5$]{e4}}]\label{frob.psl}
Let $G$ be a finite group and $\frac{G}{R(G)}\cong \mathrm{PSL}_{2}(q)$, where  $q=u^{f}\geqslant 11$, $u$ is a prime and $f\geqslant 1$ is an integer. Also let $\theta\in\mathrm{Irr}(R(G))$, $I:=I_{G}(\theta)$ and $N:=\frac{I}{R(G)}$ be a Frobenius group whose kernel is an elementary abelian $u$-group. Then either $\theta$ is extendible to $I$ and $\mathrm{cd}(G|\theta)=\{\theta(1) |G:I|, \theta(1) b\}$, for some positive integer $b$ divisible by $\frac{q^2-1}{gcd (2, q-1)}$, or $\theta$ is not extendible to $I$ and all character degrees in $\mathrm{cd}(G|\theta)$ are divisible by $u (q+1) \theta(1)$.\\
\end{lemma}
Now we end this section with the proof of Proposition B.
\begin{proof}[Proof of Proposition B]
We use the classification of finite non-abelian simple groups. If $S$ is a sporadic simple group or an alternating group, then as $\pi(\mathrm{Out}(S))\subseteq \{2\}$, we have nothing to prove. Hence we assume that $S$ is a simple group of lie type over a field of order $u^f$, for some prime $u$ and positive integer $f$. Using Tables $5$ and $6$  in \cite{atlas}, it is easy to see that $\pi(\mathrm{Out}(S))\setminus\pi(S)=\pi(f)\setminus\pi(S)$ and $u (u^f -1) \, |\, |S|$. We claim that $|\pi (u^{f}-1)| \geqslant |\pi (f)|$. If $f=1$, then we have nothing to prove. Thus assume that $f\neq 1$ and $|\pi(f)|=r$, where $r$ is a positive integer. Then using Zsigmondy's theorem, it is clearly seen that $|\pi(u^{f}-1)|\geqslant r$ and we obtain the claim. Hence $| \pi(\mathrm{Out}(S))\setminus\pi(S)| = |\pi(f)\setminus\pi(S)| \leqslant |\pi(u^f -1)| < |\pi(u (u^f -1))|\leqslant |\pi(S)|$. This completes the proof.
\end{proof}

\section{A Forbidden graph}\label{sec3}
In this section we introduce an $n$-regular simple graph, which is denoted by $F(n)$, where $n\geqslant 3$ is an integer. We show that this graph can not occur as the character-graph of any finite group.

Let $n\geqslant 3$ be an integer, $W_1:=\{1, \ldots , n\}$ and $W_2:=\{1', \ldots , n'\}$. We define the graph $F(n)$ to be a  simple graph with vertex set $V(F)$ and edge set $E(F)$, where $V(F)=W_1\dot{\cup} W_2$ in which $F(n)[W_1]\cong F(n)[W_2]\cong K_n$ and for each $i\in\{1, 2, \ldots , n\}=W_1$, $i$ is precisely adjacent to $i'\in W_2$.
\begin{theorem}\label{forbidengraph}
The graph $F(n)$, where $n\geqslant 3$, can not occur as the character-graph $\Delta(G)$ of any finite group $G$.
\end{theorem}
We wish to prove Theorem \ref{forbidengraph}.  On the contrary suppose that  $G$ is a finite group such that $\Delta(G)\cong F(n)$, for some positive integer $n\geqslant 3$. Then we can partition $\rho(G)=V_1\dot{\cup} V_2$ where $V_1:=\{ {a_{1}}_{1}, {a_{1}}_{2}, \ldots , {a_{1}}_{n}\}$, $V_2:=\{ {a_{2}}_{1}, {a_{2}}_{2}, \ldots , {a_{2}}_{n}\}$, $\Delta(G)[V_1]\cong\Delta(G)[V_2]\cong K_n$ and for every $s\in\{1,2, \ldots , n\}$, the vertex  ${a_{1}}_{s}\in V_1$ is precisely adjacent to ${a_{2}}_{s}\in V_2$.  As $ \Delta(G)$ is an $n$-regular character-graph with $2n$ vertices, using \cite[Theorem $A$]{morre.reg}, $G$ is a non-solvable group.

\begin{lemma}\label{cardgraph}
 $|\rho(G)|\geqslant 8$.
\end{lemma}
\begin{proof}
 If $|\rho(G)|< 8$, then $|\rho(G)|=6$ and as $\Delta(G)$ is a $3$-regular graph, we have a contradiction with  \cite[Theorem $A$]{tongreg}. 
\end{proof}

\begin{lemma}\label{claim}
 Let  $S:=\frac{M}{R(G)}$ be a simple chief factor of $G$. If $\pi(S)\subseteq V_i$ for some $i\in\{1,2\}$, then $G$ does not exist.
\end{lemma}
\begin{proof}
Set $\frac{C}{R(G)}:=C_{\frac{G}{R(G)}} \Big(\frac{M}{R(G)} \Big)$. Then it is easy to see  that $S\cong \frac{MC}{C} \lhd \frac{G}{C} \leq \mathrm{Aut}(S)$. Now let $x\in \rho(C)\setminus\pi(S)$. Thus applying Lemma \ref{notri} for  $\frac{MC}{C}$ and using Gallagher's theorem,
  $x$ is adjacent to at least two distinct primes in $\pi (S)$. Therefore by the structure of $F(n)$, it is clear that 
   $\rho (C)\cup \pi(S)\subseteq V_{i}$ and  $V_j \subseteq \pi \Big( \frac{G}{C}\Big) \setminus (\rho (C) \cup \pi (S) )\subseteq \pi(\mathrm{Out}(S))\setminus\pi(S)$, where $\{i,j\}=\{1,2\}$. Hence using Proposition $B$,  $n=|V_{j}| \leqslant |\pi (\mathrm{Out}(S))\setminus \pi (S) | < |\pi (S)| \leqslant |V_i|=  n$. This is a contradiction and thus $G$ does not exist.
\end{proof}

\begin{lemma}\label{almost}
There exists a normal subgroup $R(G)\lneq M \unlhd G$, such that $\frac{G}{R(G)}$ is an almost simple group with socle  $S:=\frac{M}{R(G)}$.
\end{lemma}
\begin{proof}
Let $\frac{M}{R(G)}$ be a chief factor of $G$. Then there exists a non-abelian simple group $S$ and some positive integer $k$ such that $\frac{M}{R(G)}\cong S^k$. Set $\frac{C}{R(G)}:=C_{\frac{G}{R(G)}} \Big( \frac{M}{R(G)}\Big)$. It is easy to see that $S^{k}\cong \frac{MC}{C} \vartriangleleft \frac{G}{C} \leq \mathrm{Aut(S^{k})}$.

We claim that $k=1$. On the contrary suppose that $k > 1$. Hence by \cite[Main Theorem]{lewis.characteristically},  $\Delta\Big(\frac{G}{C}\Big)$ is a complete graph with at least three vertices. Then using 
 the structure of $F(n)$,  we can assume that $\pi\Big( \frac{G}{C} \Big) \subseteq V_{i}$ and so $ V_{j}\subseteq \rho(G) \setminus\pi\Big( \frac{G}{C} \Big)$, where $\{i,j\}=\{1,2\}$.  Let $r\in V_{j}$. Hence there exists $\theta\in\mathrm{Irr}(C)$ such that $r | \theta(1)$. Also there exists $L\unlhd MC$ such that $\frac{L}{C}\cong S$. Thus using Lemma \ref{notri}, either for some $\chi\in\mathrm{Irr}(L|\theta)$, $\frac{\chi(1)}{\theta(1)}$ is divisible by two distinct primes  $p, q \in \pi (S)$  or $\theta$ is extendible to $\theta_{0}\in\mathrm{Irr}(L)$ and $S \cong A_{5}$ or $\mathrm{PSL}_{2}(8)$. If the later case holds then by Gallagher's theorem $r$ is adjacent to  all vertices in $\pi(S)$. Therefore, the induced subgraph of $\Delta(G)$ on either  $\{r, p, q\}$ or  $\{r, 2,3\}$ is a triangle. It is a contradiction with the structure of $F(n)$. Hence $k=1$ and $\frac{M}{R(G)}\cong S$ is a simple group.

Now we claim that $C=R(G)$. On the contrary suppose that $C\neq R(G)$. Choose $R(G)\lneq L\leq C$ such that $\frac{L}{R(G)}$ is a chief factor of $G$. Hence similar to the above paragraph, $\frac{L}{R(G)}\cong T$, for some non-abelian simple group $T$. Thus $\Delta\big(\frac{ML}{R(G)}\big)\cong \Delta(S\times T)\subseteq\Delta(G)$. Since $2\in\pi(S)\cap\pi(T)$, using \cite[Theorem $3.1$]{directcharactergraph}, for distinct primes  $s\in \pi(S)\setminus \{2\}$  and  $t\in\pi(T)\setminus \{2\}$,   $\Delta(S\times T)[\{2,s,t\}]\cong K_3$. Hence by the structure of $F(n)$, it is clear that  $\pi(S)\subseteq \rho(S\times T) \subseteq V_i$, for some $i\in \{1,2\}$. Thus using Lemma \ref{claim}, we have a contradiction. Therefore $C=R(G)$ and this completes the proof.
\end{proof}
In the sequal of this section, using Lemma \ref{almost}, we have $\frac{G}{R(G)}$ is an almost simple group with socle $S:=\frac{M}{R(G)}$, where $R(G)\lneq M \unlhd G$ is a normal subgroup of $G$. 

\begin{lemma}\label{rovi}
Suppose that for some $p\in \pi(S)$,  $\pi(S)\setminus \{p\}\subseteq V_{i}$ and  $| (\rho(R(G))\setminus\pi(S)) \cap V_j|\geqslant 3$, where  $\{i,j\}=\{1,2\}$. Then $G$ does not exist.
\end{lemma}
\begin{proof}
If $p\in V_i$, then by Lemma \ref{claim} we have nothing to prove. Hence we can assume that $p\in V_j$. As $| (\rho(R(G))\setminus\pi(S)) \cap V_j|\geqslant 3$, using \cite[Theorem $A$]{a1}, there exist $\theta\in\mathrm{Irr}(R(G))$  and $a,b\in (\rho(R(G))\setminus\pi(S)) \cap V_j$ such that $ab | \theta (1)$. Thus using  Lemma \ref{notri} for $\frac{M}{R(G)}$ and Gallagher's theorem, there exists $c\in\pi(S)\setminus\{p\}$, such that $\Delta(G)[\{a,b,c\}]$ is a triangle. It is a contradiction with 
the structure of $F(n)$ and this completes the proof.
\end{proof}
 
\begin{lemma}\label{removes}
Suppose that either $\Delta(S)$ is connected and has no cut edges, or $S$ is isomorphic to one of the groups $\mathrm{M}_{11}$, $\mathrm{PSL}_{3}(4)$ or $^{2}\mathrm{B}_{2}(q^2)$, where $q^2=2^{2m+1}$ for some integer $m\geqslant 1$ and  $r:=q^2-1$ is a Mersenne prime. Then $G$ does not exist.
\end{lemma}
\begin{proof}
If $\Delta(S)$ is connected and has no cut edges, then by \cite{whitesim} and Lemma \ref{simcutedge}, it is clear that $\pi(S)\subseteq V_i$, for some $i\in\{1,2\}$. It is a contradiction with Lemma \ref{claim}. Hence we can assume that $S$ is isomorphic to one of groups $\mathrm{M}_{11}$, $\mathrm{PSL}_{3}(4)$ or $^{2}\mathrm{B}_{2}(q^2)$, where $q^2=2^{2m+1}$ for some integer $m\geqslant 1$ and  $r:=q^2-1$ is a Mersenne prime. Looking at the structure of $\Delta(S)$ \cite{whitesim}, there exists $x\in \pi(S)$ such that $\Delta(G)[\pi(S)\setminus\{x\}]$ is a complete graph with at least three vertices. Set $A:=\pi(S)\setminus\{x\}$. By \cite{atlas},  $|\pi(\mathrm{Out}(S))\setminus\pi(S)|\leqslant 1$. Note that when $|\pi(\mathrm{Out}(S))\setminus\pi(S)| = 1$, then    $S\cong$  $^{2}B_{2}(q^2)$ and using \cite[Theorem $1.1$]{ghafar}, $2m+1$ is adjacent in $\Delta(G)$ to all primes in $A$. Hence the structure of $F(n)$  implies that $A\cup\pi(|G:M|)\subseteq V_i$, for some $i\in\{1,2\}$.  Thus using Lemmas \ref{cardgraph} and \ref{rovi}, we are done.
\end{proof}
Applying Lemmas  \ref{simcutedge} and \ref{removes}, the remaining case is $S\cong \mathrm{PSL}_{2}(q)$, where $q=u^{f}$, for some prime $u$ and positive integer $f$. Note that we frequently use Lemmas \ref{deltaS}, \ref{cardgraph} and the structure of $F(n)$ without any further references.
\begin{lemma}\label{k3graph}
If $|\pi(S)|=3$, then $G$ does not exist.
\end{lemma}
\begin{proof}
Using {\cite{herzog}}, it is clear that $\pi\Big(\frac{G}{R(G)}\Big) = \pi(S)$. Hence as $|\rho(G)|\geqslant 8$,  by Lemma \ref{claim} there exists $i\in\{1,2\}$ such that $|\pi(S)\cap V_i |=1$ and  $|(\rho(R(G))\setminus\pi(S))\cap V_i|\geqslant 3$. Thus using Lemma \ref{rovi}, we are done.
\end{proof}

\begin{lemma}
If $|\pi(S)|\geqslant 4$ and $\overline{\Delta(S)}$ is bipartite, then $G$ does not exist.
\end{lemma}
\begin{proof}
Obviously,  $q$ is odd, and for some $\epsilon\in\{\pm 1\}$, $\pi(q+\epsilon)=\{2\}$ and $|\pi(q-\epsilon)|\geqslant3$. Hence using the structure of $F(n)$, Lemma \ref{index} implies that $(\pi(S)\setminus \{u\}) \cup \pi(|G:M|)\subseteq V_i$, for some $i\in\{1,2\}$. Thus as   $|\rho(G)|\geqslant 8$, using Lemma \ref{rovi}, we have nothing to prove.
\end{proof}

\begin{lemma}\label{starlemma}
Let for some $i\in\{1,2\}$ and $s\in\{1,2,\ldots , n\}$,  ${a_{i}}_{s}\in (\rho(R(G))\setminus\pi(S)) \cap V_i$ and ${a_{i}}_{s}, {a_{j}}_{s}\notin \pi(|G:M|)$, where $\{i,j\}=\{1,2\}$. Also suppose that either:\\
{\bf 1.} $q$ is odd, $|\pi(q\pm 1)|\geqslant 2$ and for some $\epsilon\in\{\pm 1\}$, $\pi(q+\epsilon)\subseteq V_j$, or\\
{\bf 2.} $q$ is even and for some $\epsilon\in\{\pm 1\}$, $|\pi(q+\epsilon)|\geqslant 2$ and $\pi(q+\epsilon)\subseteq V_j$.\\
Then $G$ does not exist.
\end{lemma}
\begin{proof}
By assumption, there exists $\chi\in\mathrm{Irr}(G)$ such that ${a_{i}}_{s} {a_{j}}_{s}\, |\, \chi(1)$. Let $\phi\in\mathrm{Irr}(M)$ be a constituent of $\chi_{M}$ and $\theta\in\mathrm{Irr}(R(G))$ be a constituent of $\phi_{R(G)}$. As ${a_{i}}_{s}, {a_{j}}_{s}\notin \pi(|G:M|)$, Lemma \ref{11.29} implies that ${a_{i}}_{s} {a_{j}}_{s} | \phi(1)$. Hence as ${a_{i}}_{s} \notin \pi(S)$, again using Lemma \ref{11.29}, ${a_{i}}_{s}  | \theta(1)$. Let $I:=I_{M}(\theta)$ and $N:=\frac{I}{R(G)}$. We use Dickson's list for the structure of $N$. First assume that $N=S$. Then as $\mathrm{SL}_{2}(u^f)$ is the Schur representation group of $S$, for some $m\in\mathrm{cd}(M|\theta)$ we have ${a_{i}}_{s} (q+\epsilon) | m$. Hence  $\Delta(G)[\pi(m)]$ is a complete graph with at least three vertices. It is a contradiction. Thus $N\lneq S$ and one of the following cases occurs. Set $d:=gcd (2,q-1)$.\\
{\it Case 1) } $N$ is an elementary abelian $u$-group. Then by Clifford's theorem, for some $y\in \pi\Big (\frac{q^2-1}{d}\Big)\setminus \{{a_{j}}_{s}\}$, we have ${a_{i}}_{s} {a_{j}}_{s} y | \phi(1)$ which is a contradiction.\\
{\it Case 2) } $N$ is contained in a dihedral group. Then using Clifford's theorem, for some $\delta\in\{\pm 1\}$ there exists $y\in \pi\Big (u\frac{q+\delta}{d}\Big) \setminus \{{a_{j}}_{s}\}$, such that ${a_{i}}_{s} {a_{j}}_{s} y | \phi(1)$. It is again a contradiction.\\
{\it Case 3) } $N$ is a Frobenius group whose kernel is an elementary abelian $u$-group. Then using Lemma \ref{frob.psl}, either ${a_{i}}_{s} {a_{j}}_{s} y | \phi(1)$, for some $y\in \pi(u(q+1))\setminus\{{a_{j}}_{s}\}$,  or we have ${a_{i}}_{s} \Big( \frac{q^2-1}{d} \Big) | m$, for some $m\in\mathrm{cd}(M|\theta)$. It is a contradiction with the structure of   $F(n)$.\\
{\it Case 4) } $N\cong A_5$. Then as $\mathrm{SL}_{2}(5)$ is the Schur representation group of $A_5$, using Clifford's theorem one of the following occurs:\\
{\it a) }  $q$ is even. Then for some $m\in\mathrm{cd}(M|\theta)$, $m$ is divisible by $\theta(1) \frac{q(q^2-1)}{20}$. Hence for some $y\in \pi(q+\epsilon)\setminus\{5\}$, we have $2 {a_{i}}_{s} y | m$. Thus $\Delta(G)[\{2, {a_{i}}_{s}, y\}]\cong K_3$ which is a contradiction. \\
{\it b) } $q$ is odd. Then for some $m\in\mathrm{cd}(M|\theta)$, $m$ is divisible by $\theta(1) \frac{q(q^2-1)}{30}$. Hence $2 {a_{i}}_{s} y | m$, for some $y\in \pi(u(q^2-1))\setminus \{2,3,5\}$. This is a contradiction.\\
{\it Case 5) } Either $N\cong A_4$ and $8| q^2-1$  or $N\cong S_4$ and $16 | q^2-1$. Then using Clifford's theorem $\theta(1) \frac{q(q^2-1)}{2 |N|} \Big| \phi(1)$. Hence as $|\pi(S)|\geqslant 4$, there exists $y\in \pi(S)\setminus\{2,3,{a_{j}}_{s}\}$ such that ${a_{i}}_{s} {a_{j}}_{s} y | \phi(1)$. This is again a contradiction.\\
{\it Case 6) } $N\cong \mathrm{PSL}_{2}(u^m)$, where $m\neq f$ is a positive divisor of $f$, $u^m\geqslant 7$ and $u^m\neq 9$. Then as $\mathrm{SL}_{2}(u^{m})$ is the Schur representation group of $N$, using Clifford's theorem $l:=\theta(1) u^{f-m} (u^{m}+1) \frac{(u^f+1)(u^f-1)}{u^{2m}-1}\in\mathrm{cd}(M|\theta)$. Thus by Zsigmondy's theorem there exists $y\in \pi(q+\epsilon)\cap \pi\Big(\frac{u^{2f}-1}{u^m -1} \Big)$ such that ${a_{i}}_{s} uy\, |\, l$ which is a contradiction.\\
{\it Case 7) } $N\cong \mathrm{PSL}_{2}(9)$, such that $u=3$ and  $f\geqslant 4$ is even. Then by Clifford's theorem $\phi(1)$ is divisible by $\theta(1) \frac{3(q^2-1)}{80}$.  Hence as $|\pi(S)|\geqslant 4$, there exists $y\in \pi(S)\setminus\{2,5,{a_{j}}_{s}\}$ so that ${a_{i}}_{s} {a_{j}}_{s} y | \phi(1)$, which is impossible.\\
{\it Case 8) } $N\cong \mathrm{PGL}_{2}(u^m)$, where $u$ is odd and $2m$ is a positive divisor of $f$. Then using Clifford's theorem $\phi(1)$ is divisible by $\theta(1) |S:N|=\theta(1)  u^{f-m} (u^f+1) \frac{(u^{f}-1)/2}{(u^{2m}-1)}$. Hence as $2m | f$, by Zsigmondy's theorem we deduce that there exists $y\in\Big(\pi(q+\epsilon)\cap\pi\Big( \frac{(u^f+1) (u^f-1)}{u^{2m}-1} \Big)\Big)\setminus\{2\}$ such that ${a_{i}}_{s} uy | \phi(1)$. This is a contradiction and completes the proof.
\end{proof}

\begin{lemma}\label{star2}
Let for some $\epsilon\in\{\pm 1\}$, $|\pi(q+\epsilon )|\geqslant 2$. Also let  $\pi(q+\epsilon )\cup \pi(|G:M|)\subseteq V_j$ and  $(\rho(R(G))\setminus\pi(S))\cap V_i\neq\emptyset$, where $\{i,j\}=\{1,2\}$. Then $G$ does not exist.
\end{lemma}
\begin{proof}
Let  ${a_{i}}_{s}\in (\rho(R(G))\setminus\pi(S))\cap V_i$, for some  $s\in\{1, 2, \ldots , n\}$. If $\pi(q-\epsilon)\subseteq V_j$, then  $(\pi(S)\setminus\{u\})\cup\pi(|G:M|)\subseteq V_j$, and  as   $|\rho(G)|\geqslant 8$, by Lemma \ref{rovi} we are done. Thus we can assume that $\pi(q-\epsilon )\cap V_i\neq \emptyset$. Let $v\in \pi(q-\epsilon )\cap V_i$. If ${a_{j}}_{s}\in\pi(|G:M|)$, then using Lemma \ref{index}, ${a_{j}}_{s}$ and $v$ are adjacent vertices in $\Delta(G)$  which is a contradiction with the choice of ${a_{j}}_{s}$. Hence ${a_{j}}_{s}\notin \pi(|G:M|)$ and as ${a_{i}}_{s}\notin\pi(|G:M|)$,  
Lemmas \ref{rovi} and \ref{starlemma} 
complete the proof.
\end{proof}

\begin{lemma}\label{odd}
If $q\neq 5$ is odd and $\overline{\Delta(S)}$ is non-bipartite, then $G$ does not exist.
\end{lemma}
\begin{proof}
It is clear that $|\pi(q\pm 1)|\geqslant 2$. We claim that for some $j\in\{1,2\}$ and $\epsilon\in\{\pm 1\}$, $\pi(q+\epsilon )\subseteq V_j$. If for some $\epsilon\in\{\pm 1\}$,  $|\pi(q+\epsilon )|\geqslant 3$, then using the structure of $F(n)$, we obtain the claim. Hence $|\pi(q\pm 1)|=2$. Thus as $gcd(q+1,q-1)=2$, by the structure of $F(n)$, for some $j\in\{1,2\}$ and $\epsilon\in\{\pm 1\}$, we have $\pi(q+\epsilon )\subseteq V_j$, as we claimed. Hence  Lemma \ref{index} implies that $\pi(|G:M|)\cup\pi(q+\epsilon )\subseteq V_j$. If either $\pi(|G:M|)\setminus\pi(S)\neq \emptyset$ or $\pi(|G:M|)\subseteq\pi(S)$ and $|\pi(q-\epsilon)\setminus\{2\}|\geqslant 2$, then using Lemma \ref{index}, $(\pi(S)\setminus\{u\})\cup\pi(|G:M|)\subseteq V_j$. Therefore as    $|\rho(G)|\geqslant 8$, by Lemma \ref{star2}, we have nothing to prove. Hence $\pi(|G:M|)\subseteq\pi(S)$ and $|\pi(q-\epsilon)|=2$. Thus as $\pi(q+\epsilon )\cup\pi(|G:M|)\subseteq V_j$ and $|\rho(G)|\geqslant 8$, Lemma \ref{star2} completes the proof.
\end{proof}

\begin{lemma}\label{even1}
Let $q$ be even and $|\pi(S)|\geqslant 4$. Also when $\pi(|G:M|)\subseteq \pi(S)$, we have $\rho(R(G))\setminus\pi(S)\neq \emptyset$. Then $G$ does not exist.
\end{lemma}
\begin{proof}
As $|\pi(S)|\geqslant 4$, there exists $\epsilon\in\{\pm 1\}$ such that $|\pi(q+\epsilon)|\geqslant 2$. First assume that $\pi(|G:M|)\setminus\pi(S)\neq \emptyset$. Then using Lemma \ref{index}, for some $j\in\{1,2\}$, $\pi(q+\epsilon )\cup \pi(|G:M|)\subseteq V_j$ and  either for some prime $v\in V_i$, we have  $\pi(q-\epsilon ) =\{v\}$ or $\pi(q-\epsilon ) \subseteq V_j$, where $\{i,j\}=\{1,2\}$. If the first case occurs, then as   $|\rho(G)|\geqslant 8$, using lemma \ref{star2} we have nothing to prove. Suppose the later case holds. Then $(\pi(S)\setminus\{u\})\cup\pi(|G:M|)\subseteq V_j$ and  as   $|\rho(G)|\geqslant 8$,  by Lemma \ref{star2}, we are done. Now let $\pi(|G:M|)\subseteq\pi(S)$. Then by assumption, for some $i\in\{1,2\}$, $X:=(\rho(R(G))\setminus\pi(S))\cap V_i\neq \emptyset$. Let  ${a_{i}}_{s}\in X$, for some $s\in\{1, 2, \ldots , n\}$. Suppose $\{i,j\}=\{1,2\}$.  Now one of the following cases occurs: \\
{\bf 1.}  $\pi(q+\epsilon)\subseteq V_j$. Then using Lemma \ref{index},  $\pi(q+\epsilon)\cup\pi(|G:M|)\subseteq V_j$ and so by Lemma \ref{star2},  $G$ does not exist. \\
{\bf 2.} $\pi(q+\epsilon)\subseteq V_i$. If $\pi(q-\epsilon )\subseteq V_j$ and $|\pi(q-\epsilon )|\geqslant 2$, then by Lemma \ref{index}, $\pi(q-\epsilon )\cup\pi(|G:M|)\subseteq V_j$ and using Lemma \ref{star2}, we have a contradiction. Hence as $|V_j|\geqslant 4$, in the remaining cases, using Lemma \ref{star2},  $G$ does not exist. \\
{\bf 3.} $\pi(q+\epsilon )\cap V_k\neq \emptyset$, for every $ k\in\{1,2\}$. Then $|\pi(q+\epsilon )|=2$. If $|\pi(q-\epsilon )|\geqslant 2$ and for some $k\in\{1,2\}$, $\pi(q-\epsilon )\subseteq V_k$, then using cases {\bf (a)} and {\bf (b)}, we have nothing to prove. Thus we can assume that either $|\pi(q-\epsilon )|=1$ or $|\pi(q-\epsilon )|=2$ and $\pi(q-\epsilon )\cap V_k\neq \emptyset$, for every $ k \in\{1,2\}$. Note that using {\rm(}\cite[Lemma $2.4$]{four}{\rm)} and Lemma \ref{index}, $\frac{G}{R(G)}=S$. By assumption, there exists $\chi\in\mathrm{Irr}(G)$ such that ${a_{i}}_{s} {a_{j}}_{s} | \chi(1)$. Let $\theta\in\mathrm{Irr}(R(G))$ be a constituent of $\chi_{R(G)}$. Using Lemma \ref{11.29}, ${a_{i}}_{s} | \theta(1)$. Let $I:=I_{G}(\theta)$ and $N:=\frac{I}{R(G)}$. We use Dickson's list for the structure of $N$. First suppose that $N=S$. Then as the Schur multiplier of $S$ is trivial, $\theta$ is extendible to $G$. Hence using Gallagher's theorem $\Delta(G)[\{{a_{i}}_{s}\}\cup \pi(q+\epsilon)]\cong K_3$, which is a  contradiction with the structure of $F(n)$. Hence $N\lneq S$ and one of the following cases occurs:\\
{\it Case 1) } $N$ is either an elementary abelian $2$-group, contained in a dihedral group, or a Frobenius group whose kernel is an elementary abelian $2$-group. If $N$ is a Frobenius group and $\theta$ is extendible to $I$, then using Lemma \ref{frob.psl}, $\Delta(G)[\{{a_{i}}_{s}\}\cup\pi(q+\epsilon)]\cong K_3$ which is impossible. Thus in the remaining cases, by Clifford's theorem and Lemma \ref{frob.psl}, we can choose  $y\in\pi(S)\setminus\{{a_{j}}_{s}\}$ such that ${a_{i}}_{s} {a_{j}}_{s} y | \chi(1)$. It is a contradiction with the structure of $F(n)$. \\
{\it Case 2) } $N\cong A_5$. Then as $\mathrm{SL}_{2}(5)$ is the Schur representation group of $A_5$, using Clifford's theorem, for some $m\in\mathrm{cd}(G|\theta)$, $m$ is divisible by $\theta(1) \frac{q(q^2-1)}{20}$. If $|\pi(S)|=5$, there exists  $y\in (\pi(q^2-1)\setminus\{5\})\cap V_j$  such that $2 {a_{i}}_{s} y | m$  which is a contradiction. Hence $|\pi(S)|=4$ and by {\rm(}\cite[Lemma $2.4$]{four}{\rm)}, $S\cong \mathrm{PSL}_{2}(16)$. Thus $X:=\{2,3,17\}\subseteq\pi(m)$. Hence by the structure of $F(n)$ there exists $k\in\{1,2\}$ such that $X\subseteq V_k$ and  using Lemma \ref{rovi}, we have a contradiction.\\
{\it Case 3) } $N\cong \mathrm{PSL}_{2}(2^m)$, where $m\neq f$ is a positive divisor of $f$ and $2^m \geqslant 8$. Thus using   {\rm(}\cite[Lemma $2.4$]{four}, \cite[Lemma $3.2$]{e4}{\rm)},     $|\pi(q\pm 1)|=2$. Then by Clifford's theorem, $l:= \theta(1)  2^{f-m} (2^f +1) \frac{(2^f-1)}{2^{m}-1}\in\mathrm{cd}(G|\theta)$. Hence $2 {a_{i}}_{s} (2^f+1) \, |\, l$. Thus $\Delta(G)[\pi(l)]$ is a complete graph with at least three vertices. It is a contradiction.
\end{proof}
The following Lemma completes the proof of Theorem \ref{forbidengraph}.
\begin{lemma}\label{even}
Let $q$ be even and $|\pi(S)|\geqslant 4$. If $\rho(G)=\pi(S)$, then $G$ does not exist.
\end{lemma}
\begin{proof}
First we claim that $|\pi(q\pm 1)|\geqslant 3$. On the contrary suppose  $|\pi(q+\epsilon)|\leqslant 2$, for some $\epsilon\in\{\pm 1\}$. Then as $|\rho(G)|\geqslant 8$, we have $|\pi(q-\epsilon)|\geqslant 5$, and so $\pi(q-\epsilon)\subseteq V_i$ and  $V_j\subseteq\{2\}\cup\pi(q+\epsilon)$, where $\{i,j\}=\{1,2\}$.  It is a contradiction.  Hence $|\pi(q\pm 1)|\geqslant 3$, as we claimed. If $\pi(|G:M|)\neq \emptyset$, then using Lemma \ref{index}, there exists $k\in\{1,2\}$ such that $\pi(q^2-1)\cup\pi(|G:M|)\subseteq V_k$. Thus as $\rho(G)=\pi(S)$ and $|\rho(G)|\geqslant 8$, we have a contradiction. Therefore $\pi(|G:M|)=\emptyset$.  Using Lemma \ref{claim}, for some $i\in\{1,2\}$, $V_i=\{2\} \cup\pi(q+\epsilon) $ and $V_j=\pi(q-\epsilon)$. Let ${a_{i}}_{s}\in V_i$, for some  $s\in\{1, 2, \ldots , n\}$. Then for some $\chi\in\mathrm{Irr}(G)$, ${a_{i}}_{s} {a_{j}}_{s}\, |\, \chi(1)$. Let $\theta\in\mathrm{Irr}(R(G))$ be a constituent of $\chi_{R(G)}$. Also let $I:=I_{G}(\theta)$ and $N:=\frac{I}{R(G)}$. Now we consider Dickson's list on the structure of $N$. First assume that $N=S$. Then as the Schur multiplier of $S$ is trivial, $\theta$ is extendible to $G$. Thus $\mathrm{cd}(G|\theta)=\{\theta(1)\, m\, |\, m\in\mathrm{cd}(S)\}$. Therefore by the structure of $F(n)$, $\theta(1)=1$ and $\mathrm{cd}(G|\theta)=\mathrm{cd}(S)$. This is a contradiction as $\chi(1)\in\mathrm{cd}(G|\theta)$. Hence $N\lneq S$ and   one of the following cases holds:\\
{\it Case 1)} $N$ is either an elementary abelian $2$-group, contained in a dihedral group, or a Frobenius group whose kernel is an elementary abelian $2$-group. If $N$ is a Frobenius group and $\theta$ is extendible to $I$, then using Lemma \ref{frob.psl}, $\Delta(G)[\pi(S)\setminus\{2\}]$ is a complete graph with at least three vertices. It is a contradiction. Then using Clifford's theorem and Lemma \ref{frob.psl},  ${a_{i}}_{s} {a_{j}}_{s} y | \chi(1)$, for some $y\in \pi(S)\setminus \{{a_{i}}_{s},{a_{j}}_{s}\}$ which is a contradiction with the structure of $F(n)$. \\
{\it Case 2) } $N\cong A_5$. Then as $\mathrm{SL}_{2}(5)$ is the Schur representation group of $A_5$, using Clifford's theorem, for some $m\in\mathrm{cd}(G|\theta)$, $m$ is divisible by $\theta(1) \frac{q(q^2-1)}{20}$. Hence there exists $y\in (\pi(q-\epsilon)\setminus\{5\})\cap V_j$ such that $\Delta(G)[\{y\}\cup\pi(q+\epsilon)]$ is a complete graph with at least three vertices. It is a contradiction with the structure of $F(n)$.\\
{\it Case 3) } $N\cong \mathrm{PSL}_{2}(2^m)$, where $m\neq f$ is a positive divisor of $f$ and $2^m \geqslant 8$. Then using Clifford's theorem, there exists $l:=\theta(1)  2^{f-m}   (2^f +1)  \frac{2^f-1}{2^{m}-1}\in\mathrm{cd}(G|\theta)$. Now as $m|f$, using Zsigmondy's theorem there exists $y\in \pi\Big(\frac{2^f-1}{2^{m}-1} \Big)$  such that  $\Delta(G)[2y(2^f+1)]$ is a complete graph with at least three vertices. This is again a contradiction and completes the proof.
\end{proof}

\section{Proof of  Theorem $A$}
We begin this section with the following observation.
\begin{lemma}\label{parta}
Let $G$ be a finite group with disconnected regular character-graph. Then $\Delta(G)$ is isomorphic to either $\overline{K_{3}}$ or $\overline{K_2}$.
\end{lemma}
\begin{proof}
If $|\rho(G)|\leqslant 2$, then $\Delta(G)\cong \overline{K_2}$ and we are done. Now assume that $|\rho(G)|\geqslant 3$. Using Lemma \ref{akh.reg}, $\Delta(G)\cong \overline{K_3}$ or $\Delta(G)$ has precisely two complete connected components. If the first case occurs, we have nothing to prove. Suppose the later case holds.  Then as $\Delta(G)$ is regular, by  \cite[Theorem $6.3$]{l6}, $G$ is a solvable group. Hence using \cite[Theorem $3$]{palfy}, $\Delta(G)\cong \overline{K_{2}}$, which is a contradiction. This completes the proof.
\end{proof}
Now we are ready to prove our main theorem.
\begin{proof}[Proof of Theorem A]
Let $G$ be a finite group with $k$-regular character-graph $\Delta(G)$, for some integer $k\geqslant 2$. Also let the eigenvalues of $\Delta(G)$ be in the interval $[-2,\infty )$. As $k\geqslant 2$, using Lemma \ref{parta}, 
 we assume that $\Delta(G)$ is connected. Then Lemma \ref{akh.reg} implies that $\overline{\Delta(G)}$ is bipartite. If $\Delta(G)$ is an exceptional graph, then by Lemma \ref{exceptional2}, $\overline{\Delta(G)}$ is non-bipartite which is a contradiction. Hence using \cite{cameron} and Lemma 
  \ref{reg.gen.con}, $\Delta(G)$ is either a cocktail party graph or a line graph. If $\Delta(G)$ is a cocktail party graph, then we are done. Thus we can assume that $\Delta(G)=L(\Gamma)$, for some graph $\Gamma$.
 Hence as $\Delta(G)$ is connected, the graph $\Gamma$ is connected. Also Lemma \ref{lineh} implies that  $\Gamma$ has no subgraph isomorphic to $C_{5}$ and $\alpha^{'}(\Gamma)\leqslant 2$. Now by Lemma \ref{reg.lin.con}, one of the following cases occurs:\\
{\it Case 1) } For some positive integer $r$, the  graph $\Gamma$ is $r$-regular. If $\alpha^{'}(\Gamma)=1$, then as $k\geqslant 2$, $\Delta(G)=L(\Gamma)\cong K_3$. 
 Suppose  $\alpha^{'}(\Gamma)=2$ and $m:=|V(\Gamma)|$. Also let $X:=\{x_1, x_2,x_3,x_4\}$ be the set of vertices of a maximum matching in $\Gamma$ with edges $x_1 x_2$ and $x_3 x_4$. Then  $4r=\sum_{i=1}^{4} deg (x_i )\geqslant \sum_{v\in V(\Gamma)\setminus X}\, deg (v)= (m-4)r$. Thus $ 4 \leqslant m\leqslant 8$. Note that  $\Gamma\setminus X$ is an empty graph and hence  $2 \leqslant r\leqslant 4$. If $r=2$, then using the structure of $\Gamma$, we deduce that $\Gamma\cong C_4 \cong \Delta(G)$. Thus we can assume that $r=3$ or $4$. Now one of the following cases occurs: \\
{\bf 1.} $m=8$. Then $|V(\Gamma)\setminus X|=4$. Note that every vertex $v\in V(\Gamma)\setminus X$ is adjacent to $r$ vertices in $X$. Thus there exists a vertex $x_{i}\in X$, for some $1 \leqslant i \leqslant 4$ such that $deg (x_{i})\geqslant r$ and it is a contradiction.\\
{\bf 2.} $m=7$. Then $r=4$ and $|V(\Gamma)\setminus X|=3$. Thus we obtain $\alpha^{'}(\Gamma)=3$ which is again a contradiction.\\
{\bf 3.} $m=6$. Then $|V(\Gamma)\setminus X|=2$. Hence as every vertex  $v\in V(\Gamma)\setminus X$ is adjacent to $r$ vertices in $X$, in all cases, $\Gamma$ has a subgraph isomorphic to $C_5$ which is impossible. \\
{\bf 4.} $m=5$. Then  $r=4$ and $\Gamma\cong K_5$. It is a contradiction.  \\
{\bf 5.} $m=4$. Then $r=3$ and $\Delta(G)=L(\Gamma)\cong L(K_4)\cong \mathrm{CP}(3)$. This completes the proof in this case. \\
{\it Case 2) } The  graph $\Gamma$ is a semi-regular bipartite graph with parameters $(n_1, n_2, r_1, r_2)$ where $n_1  r_1 = n_2  r_2$. Without loss of generality, assume that $1 \leqslant n_{1}\leqslant n_{2}$. Let $\alpha^{'}(\Gamma)=1$. Then it is clear that  $n_1=1$ and  $\Gamma\cong K_{1,n_2}$. Thus as $k\geqslant 2$, $\Delta(G)=L(\Gamma)\cong K_{n_2}$ where $n_2\geqslant 3$. Hence we can assume that $\alpha^{'}(\Gamma)=2$. Thus  $n_1\geqslant 2$. Let $n_1=2$. If $n_2\geqslant 3$, then we can see that $\Delta(G)=L(\Gamma)\cong L(K_{2, n_2})\cong F(n_2)$, which is a contradiction with Theorem \ref{forbidengraph}. Thus $n_2=2$ and $\Delta(G)=L(\Gamma)\cong \mathrm{CP}(2)$. If $n_1\geqslant 3$, then as $\Gamma$ is a connected semi-regular bipartite graph, we deduce that $\alpha^{'}(\Gamma)\geqslant 3$, which is a contradiction. This completes the proof.
\end{proof}

\bibliographystyle{amsplain}

\end{document}